\documentclass{amsart}

\usepackage{amsmath}
\usepackage{amssymb}
\usepackage{commath}
\usepackage{mathtools}

\usepackage{hyperref}

\theoremstyle{plain}
\newtheorem{theorem}{Theorem}
\newtheorem{lemma}[theorem]{Lemma}
\newtheorem{proposition}[theorem]{Proposition}
\newtheorem{corollary}[theorem]{Corollary}
\theoremstyle{definition}
\newtheorem{definition}[theorem]{Definition}
\newtheorem{remark}[theorem]{Remark}

\numberwithin{theorem}{section}
\numberwithin{equation}{section}

\newcommand{\B}{\mathbb{B}}
\newcommand{\D}{\mathbb{D}}
\newcommand{\C}{\mathbb{C}}
\newcommand{\HH}{\mathbb{H}}
\newcommand{\Sbb}{\mathbb{S}}

\newcommand{\R}{\mathbb{R}}
\newcommand{\T}{\mathbb{T}}

\newcommand{\M}{\mathbb{M}}

\newcommand{\U}{\mathrm{U}}
\newcommand{\SO}{\mathrm{SO}}
\newcommand{\SU}{\mathrm{SU}}
\newcommand{\Spe}{\mathrm{Sp}}

\newcommand{\im}{\operatorname{Im}}
\newcommand{\re}{\operatorname{Re}}

\newcommand{\cB}{\mathcal{B}}

\newcommand{\cF}{\mathcal{F}}
\newcommand{\cM}{\mathcal{M}}

\begin{document}

\title[Geometric structures on the quaternionic unit ball]{Geometric structures on the quaternionic unit ball and slice regular M\"obius transformations}

\author{Raul Quiroga-Barranco}
\address{Centro de Investigaci\'on en Matem\'aticas \\
	Guanajuato, M\'exico}
\email{quiroga@cimat.mx}



\begin{abstract}
	Building from ideas of hypercomplex analysis on the quaternionic unit ball, we introduce Hermitian, Riemannian and K\"ahler-like structures on the latter. These are built from the so-called regular M\"obius transformations. Such geometric structures are shown to be natural generalizations of those from the complex setup. Our structures can be considered as more natural, from the hypercomplex viewpoint, than the usual quaternionic hyperbolic geometry. Furthermore, our constructions provide solutions to problems not achieved by hyper-K\"ahler and quaternion-K\"ahler geometries when applied to the quaternionic unit ball. We prove that the Riemannian metric obtained in this work yields the same tensor previously computed by Arcozzi-Sarfatti. However, our approach is completely geometric as opposed to the function theoretic methods of Arcozzi-Sarfatti.
\end{abstract}

\maketitle

\section{Introduction}\label{sec:intro}
The hyperbolic spaces lie at the core of geometry and provide some of the most fundamental objects. Among them, the lower dimensional ones are particularly interesting. The $2$-dimensional real hyperbolic space yield the $1$-dimensional complex hyperbolic space, and the $4$-dimensional real hyperbolic space coincides with the $1$-dimensional quaternionic hyperbolic space. It is well known that these occurrences are a consequence of some algebraic facts. They come from the properties of the complex numbers $\C$ and the quaternionic numbers $\HH$. They are also related to the local isomorphisms of Lie groups $\SO(2,1) \simeq \SU(1,1)$ and $\SO(4,1) \simeq \Spe(1,1)$.

The fact that the $2$-dimensional real hyperbolic geometry is essentially the $1$-dimensional complex hyperbolic geometry is very much related to the holomorphic function theory associated to the complex plane. For example, the holomorphic isometries of the latter are precisely the orientation-preserving isometries of the former. The use of holomorphic functions also bring into play the notion of a K\"ahler form and the corresponding symplectic geometry.

One might expect that the coincidence between the $4$-dimensional real and $1$-dimensional quaternionic hyperbolic geometries can also be related to a suitable function theory supported by the quaternions $\HH$. However, it is well known that a naive straightforward generalization of holomorphicity from $\C$ to $\HH$ does not yield the desired results, mainly due to the non-commutativity of $\HH$. Nevertheless, it has also been shown that the introduction of appropriate notions of regularity or ``hyperholomorphicity'' produce very useful function theories on $\HH$. In this work we will focus on the notion of slice or Cullen regularity deeply studied in \cite{GentiliStruppa2007} (see also \cite{ColomboSabadiniStruppaFunctionalBook,GentiliStoppatoStruppa2ndEd}). Besides allowing to generalize many of the classical complex theorems, slice regularity has a natural characterization: on unit balls centered at the origin, slice regular functions are precisely those with a power series expansion with right coefficients. Furthermore, the current literature shows that a mature and very well developed so-called hypercomplex analysis is readily available through the use of slice regularity and their extensions. As an example, we refer to \cite{ColomboSabadiniEtAlSliceMonogenic,ColomboSabadiniStruppaFunctionalBook} for corresponding extensions in the case of Clifford algebras. It is also important to mention \cite{BisiStoppatoSchwarz}, where results for self-maps of the complex unit disc are extended to the quaternionic unit ball. This kind of developments are at the heart of this work.

On the other hand, the isometries of the $1$-dimensional quaternionic hyperbolic geometry are given by suitable linear fractional transformations of $\HH$, which turn out to be, in general, non-regular. Related to this fact, it was proved in \cite{BisiGentiliGeometryHUnitDisc} that the geometry on the quaternionic unit ball is not isometric to the Kobayashi geometry on the complex $2$-dimensional unit ball in $\C^2$. In other words, there is a marked incompatibility between the quaternionic hyperbolic geometry with $2$-dimensional complex geometries as well. 

Nevertheless, using the natural so-called regular product, \cite{StoppatoMobius} developed the notion of regular linear fractional transformations, also referred as regular M\"obius transformations. These are the natural ones to consider from the viewpoint of hypercomplex analysis. Hence, on the unit ball $\B$ of $\HH$ centered at the origin we have two sorts of M\"obius transformations, regular and non-regular, while $\B$ carries the metric realizing both the $4$-dimensional real and $1$-dimensional quaternionic hyperbolic geometries. The non-regular M\"obius transformations preserve such hyperbolic geometry (see \cite{BisiGentiliMobius,CaoParkerWang,ChenGreenberg}) but does not relate properly to hypercomplex analysis. The regular M\"obius transformations do relate nicely to the latter but, as shown in \cite{BisiStoppatoMobius}, regular M\"obius transformations do not preserve the hyperbolic geometry. We thus arrive to the seemingly impossibility of studying a hypercomplex hyperbolic function theory. However, \cite{BisiStoppatoMobius} proposed the problem of a further study of regular M\"obius transformations and the geometry of $\B$. As a matter of fact, it would be interesting to study as well the Clifford algebra case considered in \cite{ColomboKraussharSabadiniSymmetries}. However, our focus will lie in the quaternionic unit ball.

We note that \cite{ArcozziSarfatti} has already introduced a Riemannian metric that relates to regular functions. The corresponding geometry was constructed by considering a quaternionic Hardy space, its corresponding reproducing kernel and a pseudo-hyperbolic distance; the latter mirrors the complex case. Finally, it is proved in \cite{ArcozziSarfatti} that such distance yields indeed a Riemannian metric. Furthermore, several non-trivial and interesting geometric properties for this geometry are obtained. On the other hand, the problem of building symplectic and/or Hermitian geometries for $\B$ are not considered in \cite{ArcozziSarfatti}.

The goal of this work is to present a solution to the problems formulated so far on the analysis and geometry of $\B$. We introduce in section~\ref{sec:SliceGeometryOnB} what we call slice geometric structures on $\B$, which include Hermitian, Riemannian and K\"ahler-like geometries of a truly quaternionic nature (see Definition~\ref{def:HermRiemSympRegularOnB}). These are constructed out of regular M\"obius transformations in a similar way as the corresponding structures on the unit disk on the complex plane are obtained from (complex holomorphic) M\"obius transformations. Our slice geometric structures are computed explicitly and some fundamental properties are obtained in section~\ref{sec:PropertiesGeometryOnB}, see Theorems~\ref{thm:SGeomStruct}, \ref{thm:SliceRiemannian} and \ref{thm:SliceKahler}, thus verifying that they satisfy their claimed nature.

It is important to note some relevant facts on our methods. We prove in Theorem~\ref{thm:GeometryMobiusNonRegular} that the non-regular M\"obius transformations leave invariant (essentially) only the $1$-dimensional quaternionic hyperbolic metric, which we denote by $\widehat{G}$, and that there are no Hermitian or K\"ahler-like structures on $\B$ invariant under such non-regular maps. This is one of the reasons for our approach: to consider regular M\"obius transformations and replace the metric $\widehat{G}$ with some other geometric structures related to slice regularity. We have effectively used hypercomplex analysis as the guiding light to determine the most adequate geometry on $\B$.

We believe our approach leads to very reasonable alternative structures on $\B$. For example, Theorem~\ref{thm:SliceRiemannian} shows that our slice Riemannian metric $G$ can be seen as a perturbation of the hyperbolic Riemannian metric $\widehat{G}$. We have also shown in Corollary~\ref{cor:g-on-DI} and Theorem~\ref{thm:SliceKahler} that both our slice Riemannian and slice K\"ahler structures restrict to their usual classical complex counterparts on the slices of $\B$ through the origin, which are unit disks in corresponding complex planes. 

Another important feature appears in our development: the Riemannian metric from \cite{ArcozziSarfatti} and our slice Riemannian metric turn out to be the same exact tensor, as it is proved in Theorem~\ref{thm:G=ArcozziSarfattiwidetildeG}. As noted above, the results from \cite{ArcozziSarfatti} make strong use of function space theoretic methods (see section~\ref{sec:PropertiesGeometryOnB} for some details). Our approach is almost exclusively based on transformations that mirror the group theoretic methods of the complex case. For the quaternionic unit ball $\B$ the regular M\"obius transformations is not a group, but we bypass the difficulties involved by taking advantage of the known fundamental properties of such transformations (see~\cite{ColomboSabadiniStruppaFunctionalBook,GentiliStoppatoStruppa2ndEd}). 

The fact that the Riemannian metrics constructed in \cite{ArcozziSarfatti} and in this work are the same reveals a fundamental property: both the function space theoretic and the geometric/group-like methods yield the same Riemannian geometry on $\B$. One can argue in favor of the latter method since it turns out to be more straightforward computation-wise. Another contribution of this work is the construction of a K\"ahler-like structure on $\B$ nicely related to the Riemannian metric that we have discussed.

On the other hand, an important feature of hypercomplex analysis of slice regular functions are the so-called representation formulas: every regular function can be determined from its values on a single slice. We prove in Corollaries~\ref{cor:RiemRepFormula} and \ref{cor:HermKahlerRepFormulas} that our three slice geometric structures, Riemannian, Hermitian and K\"ahler, satisfy similar representation formulas, i.e.~they can be recovered from a single slice. Furthermore, our formulas can be seen as invariant, in the Riemannian case, and equivariant, in the Hermitian and K\"ahler cases, identities with respect to the action by conjugation of the group of unit quaternions. The latter is by itself an interesting set of facts from the viewpoint of Lie theory.

We recall that differential geometry has already provided an approach to the quaternionic setup: hyper-K\"ahler and quaternion-K\"ahler manifolds. We discuss in Remark~\ref{rmk:ComparisonHyperKahler} the relationship between those notions and our slice regular approach. As noted there, our methods are better fit to both geometry and hypercomplex analysis on $\B$. In particular, we observe that $\B$, with the usual hyperbolic metric $\widehat{G}$, is not hyper-K\"ahler and, although it is quaternion-K\"ahler, it does not provide globally defined $2$-forms. Our approach does yield both a K\"ahler-like structure and $2$-forms associated to a Riemannian metric $G$ nicely adapted to hypercomplex analysis.

As for the distribution of this work, sections~\ref{sec:sliceregular} and \ref{sec:Mobius} introduce our two main tools: slice regularity and M\"obius transformations, both non-regular and regular. The brief section~\ref{sec:HermitianOnD} recalls the construction of the geometric structures on the complex unit disk. As described above, sections~\ref{sec:SliceGeometryOnB} and \ref{sec:PropertiesGeometryOnB} contain the main constructions and results.

The author would like to thank Cinzia Bisi and Caterina Stoppato for their input, which allowed to improve this work. The input from the anonymous reviewers have also provided very useful information that resulted on a much better version of this work.

\section{Slice regular functions}\label{sec:sliceregular}
We will denote from now on by $\HH$ the division real algebra of quaternions and by $\B$ its unit ball centered at the origin. As usual $\re$ and $\im$ will denote the real part and imaginary part functions defined on $\HH$. In particular, $\im(\HH)$ denotes the space of purely imaginary quaternions so that its subset $\Sbb$, consisting of the unitary purely imaginary quaternions, is the $2$-dimensional sphere of imaginary units of $\HH$. For every $I \in \Sbb$, we will denote by $\C_I$ the complex plane in $\HH$ whose elements are of the form $x + yI$, where $x, y \in \R$. Correspondingly, the slice of $\B$ determined by a given $I \in \Sbb$ will be denoted by $\D_I = \B \cap \C_I$, which is the unit disk in $\C_I$. We will also consider the complex plane $\C$ in the classical sense and denote by $\D$ its unit disk centered at the origin. Note that, in this work and for the case of $\C$ (not considered as contained in $\HH$), the real and imaginary parts will have their classical meaning as real numbers. This is important to keep in mind since we will have the occasion to deal with some classical complex geometric objects.

The fundamental notion of slice regularity as first formulated in \cite{GentiliStruppa2007} (see also \cite{ColomboSabadiniStruppaFunctionalBook,GentiliStoppatoStruppa2ndEd}) is introduced in the following definition. 

\begin{definition}\label{def:sliceregular}
	A function $f : \B \rightarrow \HH$ is called slice regular if the equation
	\[
		\frac{1}{2}\bigg(\frac{\partial}{\partial x} +
				I\frac{\partial}{\partial y}\bigg)f(x + yI) = 0
	\]
	is satisfied for every $I \in \Sbb$ and $x,y \in \R$ such that $x + yI \in \B$.
\end{definition}

The theory of slice regular functions on $\B$, and furthermore on suitable domains of $\HH$, has been deeply developed. We will use well known facts from this theory and refer to \cite{GentiliStoppatoStruppa2ndEd} for further details. However, we will have the occasion to provide precise references when this helps to clarify our discussion.

Every slice regular function $f : \B \rightarrow \HH$ admits a power series expansion
\[
	f(q) = \sum_{n=0}^\infty q^n a_n
\]
with coefficients in $\HH$ that converges uniformly on compact sets. Conversely, every such convergent power series yields a slice regular function on $\B$ (see~\cite[section~1.1]{GentiliStoppatoStruppa2ndEd}). This allows us to introduce the following constructions for slice regular functions on $\B$. We refer to \cite{ColomboSabadiniStruppaFunctionalBook,GentiliStoppatoStruppa2ndEd} for further details and the proofs of our claims below.

\begin{definition}\label{def:*-product}
	For $f,g : \B \rightarrow \HH$ slice regular functions admitting power series expansions
	\[
		f(q) = \sum_{n=0}^\infty q^n a_n, \quad
		g(q) = \sum_{n=0}^\infty q^n b_n
	\]
	the (slice regular) $*$-product is the function $f*g : \B \rightarrow \HH$ given by
	\[
		f*g(q) = \sum_{n=0}^\infty q^n \sum_{k+l=n}a_k b_l,
	\]
	for every $q \in \B$.
\end{definition}

It is a well known fact that the $*$-product of slice regular functions is slice regular, and also that the space of all such functions becomes a real algebra with the $*$-product. Furthermore, one can define reciprocals with respect to the $*$-product as follows. 

\begin{definition}\label{def:regularreciprocal}
	Let $f : \B \rightarrow \HH$ be a slice regular function with power series expansion given by
	\[
		f(q) = \sum_{n=0}^\infty q^n a_n.
	\]
	The regular conjugate $f^c : \B \rightarrow \HH$ is the (slice regular) function given by
	\[
		f^c(q) = \sum_{n=0}^\infty q^n \overline{a}_n.
	\]
	The symmetrization $f^s$ of $f$ is the (slice regular) function defined as $f^s = f^c*f = f*f^c$. The regular reciprocal of $f$ is the function $f^{-*} : \B \setminus Z_{f^s} \rightarrow \HH$ defined as
	\[
		f^{-*} = \frac{1}{f^s} f^c,
	\]
	where $Z_{f^s}$ is the zero set of $f^s$.
\end{definition} 

It is a fundamental fact that all the functions introduced in Definition~\ref{def:regularreciprocal} are slice regular.

\section{Linear fractional transformations}\label{sec:Mobius}
We will recall the two main notions of linear fractional transformations associated to the unit ball $\B$. This are built from the Lie groups that we now proceed to define.

Let us denote
\[
	I_{1,1} = 
	\begin{pmatrix*}[r]
		1 & 0 \\
		0 & -1
	\end{pmatrix*}
\]
the matrix that yields pseudo-Hermitian inner products of signature $(1,1)$ on both $\C^2$ and $\HH^2$. The corresponding isometry groups for those inner products~are
\begin{align*}
	\U(1,1) &= \{ A \in M_2(\C) \mid A^* I_{1,1} A = I_{1,1} \},   \\
	\Spe(1,1) &= \{ A \in M_2(\HH) \mid A^* I_{1,1} A = I_{1,1} \},
\end{align*}
called the pseudo-unitary, complex and quaternionic, respectively, Lie groups of signature $(1,1)$. We also have the Lie groups $\T$ and $\Spe(1)$ which consist of the complex and quaternionic elements, respectively, with norm $1$. In particular, both are the unit spheres centered at the origin in their corresponding division real algebras. Finally, we will denote by $\T \times \T$ and $\Spe(1) \times \Spe(1)$ the groups of $2 \times 2$ diagonal matrices with entries in $\T$ and $\Spe(1)$ in the diagonal, respectively. It is straightforward to check that $\T \times \T \subset \U(1,1)$ and $\Spe(1) \times \Spe(1) \subset \Spe(1,1)$. Moreover, both are maximal compact subgroups (see \cite{Helgason}).

We will consider the right $\Spe(1,1)$-action on $\B$ given by
\[
	F_A(q) = q\cdot A = (qc + d)^{-1}(qa + b)
\]
for every $q \in \B$ and where $A \in \Spe(1,1)$ has the expression
\[
	A = 
	\begin{pmatrix}
		a & c \\
		b & d
	\end{pmatrix}.
\]
The map $F_A$ will be called the \textbf{M\"obius transformation} of $\B$ associated to the matrix $A \in \Spe(1,1)$. The set of all such M\"obius transformations will be denoted by $\M(\B)$. We recall that $\M(\B)$ is a Lie group isomorphic to $\Spe(1,1)/\{\pm I_2\}$ (see \cite{BisiGentiliMobius}).

\begin{remark}\label{rmk:MobiusB}
	It is a known fact that the action of $\Spe(1,1)$ on $\B$ defined above realizes all linear fractional transformations, with right coefficients, that map $\B$ onto itself. We refer to \cite{BisiGentiliMobius,CaoParkerWang,ChenGreenberg} for the proof of this claim. We observe that in most of the previous works dealing with quaternionic M\"obius transformations, with some exceptions as, for example, \cite{BisiGentiliMobius,BisiStoppatoMobius}, one considers coefficients on the left which lead to a corresponding left action. We have chosen to use coefficients and corresponding action on the right to better resemble slice regular maps. Left multiplication by a constant quaternion is not slice regular but right multiplication is.
\end{remark}

On the other hand, \cite{StoppatoMobius} introduced the regular counterpart of linear fractional transformations (see also \cite{GentiliStoppatoStruppa2ndEd}). For every $A \in \Spe(1,1)$ as above the \textbf{regular M\"obius transformation} of $\B$ associated to $A$ is the (regular) map $\cF_A : \B \rightarrow \B$ given by
\[
	\cF_A(q) = (qc + d)^{-*}*(qa + b).
\]
Every regular M\"obius transformation is a homeomorphism of $\B$. However, as a consequence of the properties of the regular product, the assignment $A \mapsto \cF_A$ is not a group homomorphism. Nevertheless, these are the transformations that better fit the analysis of slice regular functions. We will denote by $\cM(\B)$ the set of all regular M\"obius transformations. In particular, $\cM(\B)$ is not a group for the usual composition of functions. However, it is interesting to note that $\cM(\B)$ is the orbit of the identity map on $\B$ for a suitable $\Spe(1,1)$-action: see~\cite[Corollary~9.17]{GentiliStoppatoStruppa2ndEd} for further details. In other words, $\cM(\B)$ does have a close relationship with a Lie group, namely $\Spe(1,1)$.

Whenever we need to clearly distinguish the transformations of $\cM(\cB)$ with those obtained from the $\Spe(1,1)$-action considered above, we will refer to the latter as non-regular M\"obius transformations.

We recall the following alternative description of the elements in $\cM(\B)$.

\begin{proposition}[\cite{StoppatoMobius}]\label{prop:StoppatoMobius}
	A map $\cF : \B \rightarrow \B$ is a regular M\"obius transformation if and only if there exist $a \in \B$ and $u \in \Spe(1)$ such that
	\[
		\cF(q) = (q\overline{a} - 1)^{-*}*(q - a)u.
	\]
	In such case, $a$ and $u$ are uniquely associated to $\cF$.
\end{proposition}

Note that we have applied a change of sign in such expression with respect to the one found in \cite{GentiliStoppatoStruppa2ndEd,StoppatoMobius}. This will be convenient for our computations.

We establish in the next result some properties for the previous expression of regular M\"obius transformations that will be useful for our purposes. Nevertheless, this result is interesting by itself.

\begin{proposition}\label{prop:FaRegularMobius}
	For every $a \in \B$, let us denote by $\cF_a$ the regular M\"obius transformation defined by $\cF_a(q) = (q\overline{a} - 1)^{-*}*(q - a)$, for every $q \in \B$. Then, $\cF_a$ satisfies
	\[
		\cF_a(q) = (q^2|a|^2 - 2q\re(a) + 1)^{-1}
				(q^2 a - q(a^2 + 1) + a),
	\]
	for every $q \in \B$. Furthermore, the following holds for every $a \in \B$.
	\begin{enumerate}
		\item $\cF_a(a) = 0$, $\cF_a(0) = a$ and $\cF_0(q) = -q$, for all $q \in \B$.
		\item $\cF_a$ is a diffeomorphism whose directional derivative at $a$ in the direction of $\alpha$ is given by
			\[
				\dif\:(\cF_a)_a(\alpha) 
					= \frac{(1-a^2)^{-1}(a \alpha a -\alpha)}%
						{1-|a|^2},
			\]
			for every $\alpha \in \HH$.
	\end{enumerate}
\end{proposition}
\begin{proof}
	Let us fix $a \in \B$ and consider the functions on $\B$ defined by the expressions
	\[
		f(q) = q \overline{a} - 1, \quad
		g(q) = q - a,
	\]
	so that we have
	\[
		\cF_a = \frac{1}{f^s} f^c * g.
	\]
	From the definitions we can compute
	\begin{align*}
		f^s(q) &= (q\overline{a} -1) * (qa - 1) 
			= q^2 |a|^2 -2q\re(a) + 1 \\
		\big(f^c * g\big)(q) &= q^2 a - q(a^2 + 1) + a,  
	\end{align*}
	for every $q \in \B$.
	
	On the other hand, we note that either $f \equiv -1$ (for $a = 0$), or we have $f(q) = 0$ if and only if $q = (\overline{a})^{-1}$. Since $|(\overline{a})^{-1}| > 1$, we conclude that $f$ does not vanish in $\B$, and so \cite[Proposition~3.10]{GentiliStoppatoStruppa2ndEd} implies that $f^s$ does not vanish in $\B$ either. Hence we obtain the formula for $\cF_a$ in the statement and also conclude that its expression is well-defined: i.e.~its denominator does not vanish in $\B$. 
	
	From the previous remarks, we can evaluate to conclude that $\cF_a(a) = 0$ and $\cF_a(0) = a$. For $a = 0$, it is straightforward to see that $\cF_0(q) = -q$, for all $q \in \B$. In particular we obtain (1).
	
	For the first claim in (2) we use Corollary~8.24 and Theorem~9.12 from~\cite{GentiliStoppatoStruppa2ndEd}. The latter implies that $\cF_a$ is a homeomorphism, in particular injective, and so the former implies that $\cF_a$ is a local diffeomorphism. It follows that $\cF_a$ is a diffeomorphism.
	
	To obtain the formula in (2) we compute as follows
	\begin{align*}
		\dif\:(\cF_a)_a(\alpha) 
			=&\; \frac{\dif}{\dif t}\sVert[2]_{t=0}
					\cF_a(a + t\alpha) \\
			=&\; \frac{\dif}{\dif t}\sVert[2]_{t=0}
				\bigg(\big((a + t \alpha)^2|a|^2 - 2(a + t \alpha)\re(a) 
						+ 1\big)^{-1} \cdot \\
				&\quad\cdot
					\big((a + t \alpha)^2 a 
						- (a + t \alpha)(a^2 + 1) + a\big)\bigg)  \\
			=&\; \bigg(\frac{\dif}{\dif t}\sVert[2]_{t=0}
				\big((a + t \alpha)^2|a|^2 - 2(a + t \alpha)\re(a) 
						+ 1\big)^{-1} \bigg) \cdot \\
				&\quad\cdot \big(a^3 - a(a^2 + 1) + a\big) \\
				&+ \big(a^2 |a|^2 - 2a\re(a) + 1\big)^{-1} \cdot \\
				&\quad\cdot \frac{\dif}{\dif t}\sVert[2]_{t=0}
						\big((a + t \alpha)^2 a 
							- (a + t \alpha)(a^2 + 1) + a\big) \\
			=&\; 0  
				+ \big(a^2 |a|^2 - a(a + \overline{a}) + 1\big)^{-1} 
					\big(
					(a\alpha + \alpha a)a - \alpha(a^2 + 1)
					\big) \\
			=&\; \big(a^2 |a|^2 - a^2 - |a|^2 + 1\big)^{-1}
					(a\alpha a - \alpha)  \\
			=&\; \big((1 - a^2) (1 - |a|^2)\big)^{-1}
					(a\alpha a - \alpha),
	\end{align*}
	which yields the stated expression. We used the chain rule in the first identity. In the computations that follow, we used Leibniz rule for the product (in $\HH$) of $\HH$-valued functions, which is justified from the bilinearity of such product.
\end{proof}

\section{Hermitian structure on $\D$}\label{sec:HermitianOnD}
In this section we recall the group theoretic construction of the Hermitian, Riemannian and K\"ahler structures on the complex unit disk $\D \subset \C$. This will provide the motivation to understand how to deal with the corresponding problem for $\B \subset \HH$.

Consider the right action of $\U(1,1)$ on the unit disk $\D$ by M\"obius transformations which is given by
\[
	z \cdot A = (zc + d)^{-1}(za + b)
\]
where $z \in \D$ and 
\[
	A =
	\begin{pmatrix}
		a & c \\
		b & d
	\end{pmatrix}
\]
belongs to $\U(1,1)$. The isotropy subgroup of this action at the origin is $\T \times \T$ which acts on $\D$~by
\[
	z\cdot(a,d) = \overline{d}za = za\overline{d},
\]
for $z \in \D$ and $a,d \in \T$. The canonical Hermitian inner product of $\C$ is given~by
\begin{equation}\label{eq:HermitianFormCAt0}
	(\alpha, \beta) \mapsto \alpha \overline{\beta},
\end{equation}
and we think of it as a Hermitian form at the tangent space of $\D$ at the origin $0$. The $\T \times \T$-action on $\D \subset \C$ is $\R$-linear and so the differential at $0$ of the action of elements in $\T \times \T$ is given by the same linear expression. Hence, the next computation shows that the Hermitian form chosen is $\T \times \T$-invariant
\begin{equation}\label{eq:TT-invariance}
	(\alpha a\overline{d}) \overline{(\beta a\overline{d})} 
		= \alpha \overline{\beta} |a|^2 |d|^2
		= \alpha \overline{\beta},
\end{equation}
which holds for $a, d \in \T$ and by the commutativity of $\C$. The next step is to translate the Hermitian form at $0$ to other points using M\"obius transformations. For a given $z \in \D$ we consider a M\"obius transformation $\varphi$ of $\D$ such that $\varphi(z) = 0$, and we define the Hermitian form $h_z$ at the tangent space of $\D$ at $z$ by
\begin{equation}\label{eq:HermitianComplex}
	h_z(\alpha,\beta) = \dif\varphi_z(\alpha)\overline{\dif\varphi_z(\beta)},
\end{equation}
for every $\alpha, \beta \in \C$. If $\psi$ is some other M\"obius transformation of $\D$ mapping $\psi(z) = 0$, then the transformation $\varphi \circ \psi^{-1}$ fixes $0$ and so we conclude the existence of $(a,d) \in \T \times \T$ such that $\varphi \circ \psi^{-1} = R_{(a,d)}$, where $R_{(a,d)}$ stands for the right action of the element $(a,d) \in \T \times \T$ over $\D$. Hence, we can rewrite this as
\begin{equation}\label{eq:phipsiForC}
	\varphi = R_{(a,d)} \circ \psi.
\end{equation}
Then, equation~\eqref{eq:TT-invariance} implies that
\[
	\dif\varphi_z(\alpha)\overline{\dif\varphi_z(\beta)} 
		= \dif R_{(a,d)}(\dif\psi_z(\alpha))\overline{\dif R_{(a,d)}(\dif\psi_z(\beta))}
		= \dif\psi_z(\alpha)\overline{\dif\psi_z(\beta)},
\]
for every $\alpha, \beta \in \C$. It follows that the Hermitian form $h_z$ considered in equation~\eqref{eq:HermitianComplex} is well defined for all $z \in \D$. A straightforward computation now yields the following well known result.

\begin{proposition}\label{prop:GeometryComplex}
	Let $h$ be the Hermitian form defined by equation~\eqref{eq:HermitianComplex}. Then, $h$ is a Hermitian metric on $\D$ with explicit expression given by
	\[
		h_z(\alpha, \beta) = \frac{\alpha \overline{\beta}}{(1 - |z|^2)^2},
	\]
	for every $z \in \D$ and $\alpha, \beta \in \C$. Its real and imaginary parts $g$ and $\omega$, respectively, yield the hyperbolic metric on $\D$ and its associated K\"ahler form. All three tensor forms $h$, $g$ and $\omega$ are $\U(1,1)$-invariant, i.e.~invariant under the M\"obius transformations of $\D$.
\end{proposition}
\begin{proof}
	The expression follows by an explicit computation of $\dif\varphi_z$ for any choice of $\varphi$ as in the previous discussion. The $\U(1,1)$-invariance can be proved easily from its definition and the fact that $\U(1,1)$ acts on $\D$.
\end{proof}

\section{Slice geometric structures on $\B$}\label{sec:SliceGeometryOnB}
Following the classical complex case described in section~\ref{sec:HermitianOnD}, we consider the $\R$-bilinear form
\begin{equation}\label{eq:HermitianFormQuatAt0}
	H_0(\alpha, \beta) = \alpha \overline{\beta},
\end{equation}
now defined for $\alpha, \beta \in \HH$. We observe that $H_0$ is $\HH$-Hermitian for the left vector space structure of $\HH$ over itself. We will call these sort of forms left $\HH$-Hermitian. Let us also denote
\begin{equation}\label{eq:Riem2FormQuatAt0}
	G_0(\alpha,\beta) = \re(\alpha\overline{\beta}), \quad
	\Omega_0(\alpha,\beta) = \im(\alpha\overline{\beta}),
\end{equation}
which are defined for $\alpha, \beta \in \HH$ as well. In this case, $G_0$ is the canonical positive definite inner product of $\HH \simeq \R^4$ and $\Omega_0$ is an $\im(\HH)$-valued anti-symmetric $\R$-bilinear form. We will think of these three forms, $H_0$, $G_0$ and $\Omega_0$, as defined on the tangent space to $\B$ at the origin $0$. Note that equation~\eqref{eq:HermitianFormQuatAt0} corresponds to the classical  complex case given by equation~\eqref{eq:HermitianFormCAt0}. Similarly, the forms~\eqref{eq:Riem2FormQuatAt0} correspond to the real part and $i$-times the imaginary part of~\eqref{eq:HermitianFormCAt0}. To clarify the point, it is important to note that $\Omega_0$ is vector-valued, while the imaginary part of \eqref{eq:HermitianFormCAt0} is real-valued.

We now consider the $\Spe(1,1)$-action on $\B$ by M\"obius transformations discussed in section~\ref{sec:Mobius}. As stated in the next result, if we look for $\Spe(1,1)$-invariant geometric structures, then we obtain a Riemannian metric but we do not get a Hermitian metric, nor a $2$-form.

\begin{theorem}\label{thm:GeometryMobiusNonRegular}
	Let $H_0, G_0, \Omega_0$ be the forms given by equations~\eqref{eq:HermitianFormQuatAt0} and \eqref{eq:Riem2FormQuatAt0}. Then, the following properties hold.
	\begin{enumerate}
		\item There exists a unique, up to a constant factor, $\Spe(1,1)$-invariant Riemannian metric on $\B$. In particular, at $0$, all of them are a constant multiple of $G_0$. Furthermore, the $\Spe(1,1)$-invariant metric whose value at $0$ is $G_0$ has the expression
		\[
			\widehat{G}_q = \frac{\re(\alpha \overline{\beta})}%
					{(1 - |q|^2)^2},
		\]
		where $q \in \B$ and $\alpha, \beta \in \HH$.
		\item There do not exist $\HH$-valued left Hermitian positive definite tensor forms on $\B$ that are also $\Spe(1,1)$-invariant.
		\item There do not exist $\im(\HH)$-valued $2$-forms on $\B$ whose value at $0$ is $\Omega_0$ that are also $\Spe(1,1)$-invariant.
	\end{enumerate}
\end{theorem}
\begin{proof}
	It is well known (see \cite{BisiGentiliMobius,Helgason}) that the right $\Spe(1,1)$-action on $\B$ is transitive. Furthermore, if we consider (as before) $\Spe(1) \times \Spe(1)$ as a diagonally embedded subgroup of $\Spe(1,1)$, then such subgroup is precisely the stabilizer of $0$ for the $\Spe(1,1)$-action. This can be used to obtain a realization of $\B$ as a homogeneous space that can be described as follows. Our main reference is \cite{Helgason} which considers left actions as opposed to our right actions. However, the standard correspondence between these two types of actions can be used to apply the results in \cite{Helgason} to our setup.
	
	We recall that the coset space
	\[
		(\Spe(1) \times \Spe(1))\backslash \Spe(1,1)
	\]
	admits a unique manifold structure such that the $\Spe(1,1)$-action on the right is smooth. This is a consequence of \cite[Theorem~4.2,~Chapter~II]{Helgason}. If we consider the map given by
	\begin{align*}
		(\Spe(1) \times \Spe(1))\backslash \Spe(1,1) 
			&\rightarrow \B \\
		(\Spe(1) \times \Spe(1))A &\mapsto F_A(0) = 0 \cdot A,
	\end{align*}
	then the previous remarks on the $\Spe(1,1)$-action show that this is a well-defined bijection. Furthermore, Theorem~3.2 and Proposition~4.3 in Chapter~II of \cite{Helgason} imply that such map is in fact a diffeomorphism, thus providing a homogeneous space realization of $\B$. We also note that such diffeomorphism is $\Spe(1,1)$-equivariant for the right actions on both domain and target.

	On the other hand, by \cite[Table~V,~page~518]{Helgason}, the pair 
	\[
		(\Spe(1,1), \Spe(1) \times \Spe(1))
	\] 
	is a so-called Riemannian symmetric pair (see~Definition in \cite[page~209]{Helgason}) of type CII. Hence, \cite[Proposition~3.4,~Chapter~IV]{Helgason} implies the existence of a $\Spe(1,1)$-invariant Riemannian metric on the quotient space 
	\[
		(\Spe(1) \times \Spe(1))\backslash \Spe(1,1),
	\] 
	which thus translates into a $\Spe(1,1)$-invariant Riemannian metric on $\B$ as a consequence of the equivariance satisfied by the diffeomorphism between these two manifolds. This proves the existence of the Riemannian metric as stated in (1). To establish uniqueness, we recall that $\Spe(1,1)$ is a simple Lie group (see~Definition in~\cite[page~131]{Helgason}) and that the $\Spe(1) \times \Spe(1)$-action on the tangent space at $0$ in $\B$ is irreducible; this follows from the fact that the symmetric pair $(\Spe(1,1),\Spe(1) \times \Spe(1))$ is listed in \cite[Table~V,~page~518]{Helgason}. Hence, by the remarks at the end of page~255 in \cite{Helgason} the Riemannian metric in $\B$ is unique, up to a constant positive factor. Also note that, for every $u, v \in \Spe(1)$ and $\alpha, \beta \in \HH$ we have
	\[
		\re(\overline{u}\alpha v \overline{\overline{u}\beta v}) 
			= \re(\overline{u}\alpha \overline{\beta} u)
			= \re(\alpha \overline{\beta}),
	\]
	thus proving that $G_0$ is $\Spe(1) \times \Spe(1)$-invariant, and so $G_0$ is unique, up to a constant factor, by the irreducibility mentioned above. To complete the proof of (1) we note that the expression for $\widehat{G}$
	has been computed in \cite{BisiGentiliMobius}: for further details see \cite[Theorem~1.9]{BisiGentiliMobius} and the remarks before this result.
	
	Next, we claim that $\Omega_0$ is not $\Spe(1) \times \Spe(1)$-invariant, and so that the same holds for $H_0$. To see this, let us assume that $u \in \Spe(1)$ satisfies $\im(\overline{u}\alpha \overline{\overline{u} \beta}) = \im(\alpha \overline{\beta})$, for all $\alpha, \beta \in \HH$. Then, by choosing $\beta = 1$ we conclude that $\im(\overline{u}\alpha u) = \im(\alpha)$, which is easily seen to imply that $\alpha u = u\alpha$, for all $\alpha \in \HH$, because $|u| = 1$. Hence, we conclude that $u = \pm 1$, thus proving our claim.
	
	Let us assume the existence of an $\Spe(1,1)$-invariant $\HH$-valued left Hermitian positive definite tensor form on $\B$ and let $T_0$ be its value at $0$. In particular, $T_0$ is an $\Spe(1) \times \Spe(1)$-invariant $\HH$-valued left Hermitian positive definite form on $\HH$. By The Basis Theorem 2.46 from \cite{HarveySpinors}, applied to the $\HH$-Hermitian symmetric type (see~\cite[page~21]{HarveySpinors}), there exists $c > 0$ such that $H_0 = c T_0$. This implies that $H_0$ is $\Spe(1) \times \Spe(1)$-invariant, which is a contradiction. This proves~(2).
	
	Finally, the non-invariance of $\Omega_0$ for the $\Spe(1) \times \Spe(1)$-action and arguments similar to those used above yield (3).
\end{proof}

\begin{remark}\label{rmk:GeometryMobiusNonRegular}
	Theorem~\ref{thm:GeometryMobiusNonRegular} can be seen as proving that non-regular M\"obius transformations do not provide invariant Hermitian and symplectic geometric structures on $\B$ of a quaternionic nature. By comparing the proof of Theorem~\ref{thm:GeometryMobiusNonRegular} with the constructions from section~\ref{sec:HermitianOnD} one can readily see that this fact is a direct consequence of the non-commutativity of $\HH$.
	
	On the other hand, Theorem~\ref{thm:GeometryMobiusNonRegular}(1) is a restatement of the well known existence and uniqueness of an $\Spe(1,1)$-invariant metric on $\B$ (see \cite{BisiGentiliMobius,Helgason}). In fact, the Riemannian metric $\widehat{G}$ above is indeed the one studied in \cite{BisiGentiliMobius}.
\end{remark}

In order to obtain Hermitian and symplectic-like geometric structures on $\B$ of a true quaternionic nature we set our sights on $\cM(\B)$, the regular M\"obius transformations. Following the ideas from section~\ref{sec:HermitianOnD} we first note that, in contrast to the arguments found in the proof of Theorem~\ref{thm:GeometryMobiusNonRegular}, the tensor $H_0$, and so $G_0$ and $\Omega_0$, is invariant under the regular M\"obius transformations that fix the origin. We describe the latter, along with some properties, in the next result, whose first part follows from \cite[Lemma~9.19]{GentiliStoppatoStruppa2ndEd}.

\begin{proposition}\label{prop:RegularIsotropy0}
	Let us denote by $\cM(\B)_0$ the set of regular M\"obius transformations of $\B$ that fix $0$. Then, the following properties hold.
	\begin{enumerate}
		\item The right action of $\Spe(1)$ on $\B$ given by
		\[
			(q,u) \mapsto qu = R_u(q),
		\]
		yields maps of $\B$ that realize $\cM(\B)_0$. In other words, $\cM(\B)_0$ consists of the set of transformations $R_u$, for $u \in \Spe(1)$. In particular, $\cM(\B)_0$ is a group isomorphic to $\Spe(1)$ acting on the right on $\B$ and fixing $0$.
		\item The differential at $0$ of the $\cM(\B)_0$-action on $\B$ leaves invariant the form $H_0$. In particular, the same holds for $G_0$ and $\Omega_0$.
	\end{enumerate}
\end{proposition}
\begin{proof}
	As noted above, (1) is a consequence of \cite[Lemma~9.19]{GentiliStoppatoStruppa2ndEd}.
	
	Since the $\cM_0(\B)$-action on $\B$ is $\R$-linear, its differential has the same expression. Hence, the trivial identity
	\[
		(\alpha u) \overline{(\beta u)} = \alpha \overline{\beta},
	\]
	which holds for every $\alpha, \beta \in \HH$ and $u \in \Spe(1)$, yields the proof of (2).
\end{proof}

\begin{remark}\label{rmk:RegularIsotropy0}
	The $\Spe(1)$-invariance of $H_0$ is the first step to build geometric structures on $\B$ out of regular M\"obius transformations. It is important to note that Proposition~\ref{prop:RegularIsotropy0} has provided a sort of isotropy at $0$ for such transformations, which turns out to be an actual group. However, for the composition of maps, $\cM(\B)$ is not a group and so at this point the full group theoretic argument used in section~\ref{sec:HermitianOnD} for the classical complex case breaks down. Nevertheless, the next result provides the analog of equation~\eqref{eq:phipsiForC} for regular M\"obius transformations. It is worthwhile to mention that the techniques from \cite{BisiStoppatoSchwarz} are similar in spirit, since such work considers regular self-maps of $\B$ that fix the origin.
\end{remark}

\begin{lemma}\label{lem:cF1cF2-u}
	Let $\cF_1$ and $\cF_2$ be regular M\"obius transformations of $\B$. If there is some $q \in \B$ such that $\cF_1(q) = \cF_2(q) = 0$, then there exists $u \in \Spe(1)$ for which we have $\cF_1 = R_u \circ \cF_2$.
\end{lemma}
\begin{proof}
	By Proposition~\ref{prop:StoppatoMobius}, there exist $a_1, a_2 \in \B$ and $u_1, u_2 \in \Spe(1)$ such that $\cF_i = R_{u_i} \circ \cF_{a_i}$,	for $i = 1, 2$, where $\cF_{a_i}$ are given as in Proposition~\ref{prop:FaRegularMobius}. In particular, we have $\cF_{a_1}(q) = \cF_{a_2}(q) = 0$, and so Proposition~\ref{prop:FaRegularMobius}(1) implies that $a_1 = a_2 = q$. Hence, $\cF_{a_1} = \cF_{a_2}$ and we conclude 
	\[
		\cF_1 = R_u \circ \cF_2, 
	\]
	for $u = u_2^{-1} u_1 \in \Spe(1)$.
\end{proof}

The next definition introduces geometric structures on $\B$ using regular M\"obius transformations. This is the analog of equation~\eqref{eq:HermitianComplex} considered for the classical complex case.

\begin{definition}\label{def:HermRiemSympRegularOnB}
	For every $q \in \B$, let us define the forms $H_q$, $G_q$ and $\Omega_q$ by
	\begin{align*}
		H_q(\alpha, \beta) &= \dif \cF_q(\alpha) 
				\overline{\dif \cF_q(\beta)}   \\
		G_q(\alpha, \beta) &= \re\big(\dif \cF_q(\alpha) 
			\overline{\dif \cF_q(\beta)}\big)   \\
		\Omega_q(\alpha, \beta) &= \im\big(\dif \cF_q(\alpha) 
			\overline{\dif \cF_q(\beta)}\big)
	\end{align*}
	for every $\alpha, \beta \in \HH$, where $\cF \in \cM(\B)$ satisfies $\cF(q) = 0$. The tensors $H$, $G$ and $\Omega$ thus obtained on $\B$ will be called the slice Hermitian metric, the slice Riemannian metric and the slice $2$-form of $\B$, respectively. We may also abbreviate their names to s-Hermitian, s-Riemannian and s-$2$-form. All three of them will be called the slice geometric or s-geometric structures of $\B$.
\end{definition}

With the constructions considered so far, it is now easy to prove that $H$, $G$ and $\Omega$ are well defined. We will also establish the first properties of the slice geometric structures of $\B$ in the next result.

\begin{theorem}\label{thm:SGeomStruct}
	The tensors $H$, $G$ and $\Omega$ from Definition~\ref{def:HermRiemSympRegularOnB} are well defined and smooth. The slice Hermitian metric $H$ has the following explicit expression
	\[
		H_q(\alpha, \beta) = 
		\frac{(1-q^2)^{-1}(\alpha - q\alpha q)%
			(\overline{\beta - q\beta q})%
			(1-\overline{q}^2)^{-1}}{(1 - |q|^2)^2},
	\]
	for every $q \in \B$ and $\alpha, \beta \in \HH$. Furthermore, the following properties are satisfied.
	\begin{enumerate}
		\item $H$ is $\R$-bilinear.
		\item $H$ is Hermitian symmetric in the sense that $H_q(\alpha, \beta) = \overline{H_q(\beta,\alpha)}$, for every $q \in \B$ and $\alpha, \beta \in \HH$.
		\item $H$ is positive definite, i.e.~for every $q \in \B$ and $\alpha \in \HH$ we have $H_q(\alpha,\alpha) \geq 0$, with equality only for $\alpha = 0$.
	\end{enumerate}
\end{theorem}
\begin{proof}
	For a given $q \in \B$, let $\cF_1, \cF_2 \in \cM(\B)$ be such $\cF_1(q) = \cF_2(q) = 0$. By Lemma~\ref{lem:cF1cF2-u}, there exists $u \in \Spe(1)$ such that $\cF_1 = R_u \circ \cF_2$. Since $R_u$ is linear it is equal to its differential and we conclude that for every $\alpha, \beta \in \HH$ we have
	\begin{align*}
		\dif\:(\cF_1)_q(\alpha) \overline{\dif\:(\cF_1)_q(\beta)}
		&= \big(\dif\:(\cF_2)_q(\alpha) u\big) 
			\overline{\big(\dif\:(\cF_2)_q(\beta) u\big)}  \\
		&= \dif\:(\cF_2)_q(\alpha) \overline{\dif\:(\cF_2)_q(\beta)}.
	\end{align*}
	This proves that all three tensors $H$, $G$ and $\Omega$ are well defined. Hence, the definition of $H$ clearly implies claims (1) and (2).
	
	For a given $q \in \B$, it is straightforward to compute the formula for $H_q$ by choosing $\cF_q$ as the transformation mapping $q$ to $0$ and using the expression for its differential at $q$ obtained in Proposition~\ref{prop:FaRegularMobius}(2). The formula for $H$ now implies the smoothness of all three tensors.
	
	Finally, to prove the positive definiteness of $H$ let us choose $q \in \B$ and $\alpha \in \HH$. Then, we clearly have
	\[
		H_q(\alpha,\alpha) 
			= \frac{|\alpha - q\alpha q|^2}{|1-q^2|^2 (1 - |q|^2)^2} \geq 0.
	\]
	If this expression vanishes, then we have $\alpha = q \alpha q$. For $q = 0$, this implies $\alpha = 0$. On the other hand, if both $q$ and $\alpha$ are non-zero, then we conclude that
	\[
		|\alpha| = |q \alpha q| = |q|^2 |\alpha| < |\alpha|.
	\]
	This contradiction proves that $\alpha = 0$ when $q \not= 0$ as well. This yields (3) and concludes the proof.
\end{proof}

\begin{remark}\label{rmk:SGeomStruct}
	We note that the $\HH$-valued tensor $H$ is not left $\HH$-Hermitian at every point. Such property of $H_0$ has been lost. However, the properties of $H_0$ have been replaced by the $\R$-bilinearity and Hermitian positivity obtained in (1) and (2) of Theorem~\ref{thm:SGeomStruct}. On the other hand, all three tensors $H$, $G$ and $\Omega$ can be considered as invariantly obtained from regular M\"obius transformations, as it follows immediately from their definition. Full invariance cannot be even considered since $\cM(\B)$ is not a group for the composition of maps. Nevertheless, we will see in the next section that some useful properties are satisfied by the slice geometric structures.
\end{remark}

\section{Properties of the slice geometric structures on $\B$}
\label{sec:PropertiesGeometryOnB}
The tensors $G$ and $\Omega$ can be given explicit expressions by simply taking the real and imaginary parts, respectively, of $H$. However, we can obtain simpler formulas. To achieve this, we introduce some notation.

For every $I \in \Sbb$, we consider the decomposition
\[
	\HH = \C_I \oplus \C_I J,
\]
where $J$ is any element of $\Sbb$ that anti-commutes with $I$. This decomposition is orthogonal for the canonical inner product of $\HH \simeq \R^4$. In particular, $\C_I J$ is the orthogonal complement for such inner product and so it does not depend on the choice of $J$. We will denote from now on by $\pi_I$ and $\pi_I^\perp$ the orthogonal projections from $\HH$ onto $\C_I$ and $(\C_I)^\perp = \C_I J$, respectively.

It follows from the proof of Theorem~\ref{thm:SGeomStruct} that $G$ is indeed a smooth Riemannian metric. The next result further provides for $G$ some useful simplified formulas. Recall that $\widehat{G}$ is the $\Spe(1,1)$-invariant metric on $\B$ from Theorem~\ref{thm:GeometryMobiusNonRegular}(1). Also, and as usual, $T_q \B$ will denote the tangent space to $\B$ at a point $q$; note that this space is naturally isomorphic (as a real vector space) to $\HH$.

\begin{theorem}\label{thm:SliceRiemannian}
	The slice Riemannian metric $G$ of $\B$ has the following expression
	\[
		G_q(\alpha,\beta) = 
			\frac{\re((\alpha - q \alpha q)%
				(\overline{\beta - q \beta q}))}%
				{|1 - q^2|^2(1 - |q|^2)^2},
	\]
	for every $q \in \B$ and $\alpha, \beta \in \HH$. Alternatively, with the notation considered above, we also have
	\begin{align*}
		G_q(\alpha,\beta) 
			&= \frac{\re(\alpha\overline{\beta})}{(1 - |q|^2)^2}
				+ 4 \frac{|\im(q)|^2 
					\re(\pi_I^\perp(\alpha) \pi_I^\perp(\beta))}%
					{|1 - q^2|^2(1 - |q|^2)^2} \\
			&= \widehat{G}_q(\alpha,\beta) 
				+ 4 \frac{|\im(q)|^2 
					\re(\pi_I^\perp(\alpha) \pi_I^\perp(\beta))}%
					{|1 - q^2|^2(1 - |q|^2)^2},
	\end{align*}
	for every $q \in \D_I$, where $I \in \Sbb$, and $\alpha, \beta \in \HH$. In particular, for every $\alpha \in T_q \B$, the length $\|\alpha\|_q$ of $\alpha$ with respect to the metric $G_q$ is given by
	\[
		\|\alpha\|_q^2 = G_q(\alpha,\alpha) 
			= \frac{|\alpha|^2}{(1 - |q|^2)^2}
				+ 4 \frac{|\im(q)|^2\re(\pi_I^\perp(\alpha)^2)}%
					{|1 - q^2|^2(1 - |q|^2)^2}
	\]
	for every $q \in \C_I$, where $I \in \Sbb$.
\end{theorem}
\begin{proof}
	We will use in this proof the following easy to check fact
	\begin{equation}\label{eq:re(ab)=re(ba)}
		\re(\alpha \beta) = \re(\beta \alpha),
	\end{equation}
	for every $\alpha, \beta \in \HH$. In particular, the first expression is now a consequence of this identity and Theorem~\ref{thm:SGeomStruct} by the following computation
	\begin{align*}
		\re((1-q^2)^{-1}&(\alpha - q\alpha q)
			(\overline{\beta - q\beta q})
					(1-\overline{q}^2)^{-1}) = \\
				&= \re((1-\overline{q}^2)^{-1} 
					(1-q^2)^{-1}(\alpha - q\alpha q)
					(\overline{\beta - q\beta q})) \\
				&= |1 - q^2|^{-2} \re((\alpha - q\alpha q)
					(\overline{\beta - q\beta q})),
	\end{align*}
	for every $q \in \B$ and $\alpha, \beta \in \HH$.
	
	We also note that for $q \in \R$, the second and third expressions for $G_q$ are obvious consequences of the first one, which was just proved. Hence, we proceed to manipulate the expression 
	\[
		\re((\alpha - q \alpha q)
			(\overline{\beta - q \beta q}))
	\]
	under the assumption that $q \in \C_I\setminus \R$ for some $I \in \Sbb$. In particular, we can consider $\pi_I$ and $\pi_I^\perp$ and the decomposition that they entitle. We first compute as follows
	\begin{align*}
		\re((\alpha - q \alpha q)
			(\overline{\beta - q \beta q}))
			&= \re(\alpha \overline{\beta}) 
				+ \re(q\alpha q \overline{q}\overline{\beta}\overline{q})
				- \re(q\alpha q \overline{\beta})  
				- \re(\alpha \overline{q}\overline{\beta}\overline{q})  \\
			&= \re(\alpha \overline{\beta}) 
				+ |q|^2 \re(q\alpha \overline{\beta}\overline{q})
				- \re(q\alpha q \overline{\beta})  
				- \re(\overline{q} \alpha \overline{q}\overline{\beta})  \\
			&= \re(\alpha \overline{\beta}) 
				+ |q|^4 \re(\alpha \overline{\beta})
				- \re(q\alpha q \overline{\beta})  
				- \re(\overline{q} \alpha \overline{q}\overline{\beta})  \\
			&= (1 + |q|^4) \re(\alpha \overline{\beta})
			- \re(q\alpha q \overline{\beta})
			- \re(\overline{q} \alpha \overline{q} \overline{\beta}),
	\end{align*}
	where we have applied \eqref{eq:re(ab)=re(ba)} on the second identity (fourth term) and third identity (second term). Let us choose $J \in \Sbb$ that anti-commutes with $I$. If we write $\alpha = \alpha_1 + \alpha_2 J$ and $\beta = \beta_1 + \beta_2 J$,
	with $\alpha_1, \alpha_2, \beta_1, \beta_2 \in \C_I$, then a direct computation yields
	\begin{align*}
		\re(q\alpha q \overline{\beta})
			&= \re(q^2\alpha_1\overline{\beta_1})
				+ |q|^2\re(\alpha_2\overline{\beta_2}), \\
		\re(\overline{q} \alpha \overline{q} \overline{\beta}) 
			&= \re(\overline{q}^2\alpha_1\overline{\beta_1})
				+ |q|^2\re(\alpha_2\overline{\beta_2}).
	\end{align*}
	Hence, replacing in the previous expression we obtain
	\begin{align*}
		\re((\alpha - q \alpha q)&
			(\overline{\beta - q \beta q})) = \\
				&= (1 + |q|^4) \re(\alpha \overline{\beta})
					- 2\re(q^2) \re(\alpha_1\overline{\beta_1})
					- 2 |q|^2 \re(\alpha_2\overline{\beta_2}) \\
				&= \big(1 - 2 \re(q^2) + |q|^4\big) \re(\alpha \overline{\beta})
					+ 2\big(\re(q^2) - |q|^2\big)
						\re(\alpha_2\overline{\beta_2}) \\
				&= |1 - q^2|^2 \re(\alpha \overline{\beta})
					- 4 |\im(q)|^2 \re(\alpha_2\overline{\beta_2}),
	\end{align*}
	which yields the second formula for $G_q$ once we observe that 
	\[
		\alpha_2\overline{\beta_2} = - \alpha_2 J \beta_2 J 
		= - \pi_I^\perp(\alpha) \pi_I^\perp(\beta).
	\]
	The third formula now follows from Theorem~\ref{thm:GeometryMobiusNonRegular}(1).
	
	Finally, the formula for the length of a tangent vector with respect to the metric $G$ follows by taking $\alpha = \beta$ in the expressions already obtained.
\end{proof}

As a consequence, we conclude that the Riemannian geometry on the disk $\D_I = \B \cap \C_I$ induced from the slice Riemannian metric $G$ of $\B$ is exactly the usual hyperbolic geometry. More precisely, we have the next result. We will denote by $g_I$ the Riemannian metric on $\D_I$ induced by the canonical identification $\D_I \simeq \D$.

\begin{corollary}\label{cor:g-on-DI}
	For every $I \in \Sbb$, the slice Riemannian metric $G$ of $\B$ restricted to the tangent bundle $T\C_I$ of $\C_I$ is $g_I$. More precisely, we have 
	\[
		G_q(\alpha,\beta) = (g_I)_q(\alpha, \beta), 
	\]
	whenever $q \in \D_I$ and $\alpha, \beta \in \C_I$. In other words, the Riemannian metric on $\D_I$ induced by the slice Riemannian metric $G$ of $\B$ is the usual hyperbolic metric under the canonical identification $\D_I \simeq \D$.
\end{corollary}

As noted in our Introduction, a Riemannian metric on $\B$ has been constructed in \cite{ArcozziSarfatti} following a function space theoretic approach. This was achieved by considering the quaternionic Hardy space on $\B$. We will now restate some of the results from \cite[Theorem~1.1]{ArcozziSarfatti} within the setup of our notation. This will allow us to consider the Riemannian metric defined in \cite{ArcozziSarfatti} and compare it with our Riemannian metric $G$. For this, we will discuss some of the notions studied in \cite{ArcozziSarfatti}. We refer to the latter for further details and claims below on this matter.

Let us denote by $H^2(\B)$, the quaternionic Hardy space, which consists of functions defined by (convergent) power series of the form
\[
	\sum_{n=0}^\infty q^n a_n,
\]
where $(a_n)_n$ is an absolutely square summable quaternionic sequence. The space $H^2(\B)$ admits the positive definite quaternionic inner product given by
\[
	\bigg\langle
		\sum_{n=0}^\infty q^n a_n,
		\sum_{n=0}^\infty q^n b_n,
	\bigg\rangle_{H^2(\B)}
	= \sum_{n=0}^\infty \overline{b}_n a_n,
\]
for which $H^2(\B)$ turns out to be a quaternionic reproducing kernel Hilbert space. Its reproducing kernel is obtained from the following elements of $H^2(\B)$
\[
	k_q(w) = \sum_{n=0}^\infty w^n \overline{q}^n,
\]
where both $q, w \in \B$. Proceeding in a way similar to the complex case, \cite{ArcozziSarfatti} defines a pseudo-hyperbolic distance given by
\begin{equation}\label{eq:deltaArcozziSarfatti}
	\delta(p,q) =
	\Bigg(
		1 - 
		\bigg|
			\bigg\langle
				\frac{k_p}{\|k_p\|_{H^2(\B)}},
				\frac{k_q}{\|k_q\|_{H^2(\B)}}
			\bigg\rangle_{H^2(\B)}
		\bigg|^2
	\Bigg)^\frac{1}{2},
\end{equation}
for every $p,q \in \B$. As the following result states, it is proved in \cite{ArcozziSarfatti} that this distance comes from a Riemannian metric. The next result is in fact a subset of  \cite[Theorem~1.1]{ArcozziSarfatti} in the setup of our current notation.

\begin{theorem}[Arcozzi-Sarfatti~\cite{ArcozziSarfatti}]
	\label{thm:ArcozziSarfatti}
	The distance $\delta$ defined in \eqref{eq:deltaArcozziSarfatti} is induced by a Riemannian metric $\widetilde{G}$. Furthermore, for every $q \in \B$ and $\alpha \in T_q\B$, the norm of $\alpha$ with respect to $\widetilde{G}_q$ is given by
	\[
		|\alpha|_q^2  
		= \widetilde{G}_q(\alpha,\alpha) 
		= \frac{|\pi_I(\alpha)|^2}{(1 - |q|^2)^2}
			+ \frac{|\pi_I^\perp(\alpha)|^2}{|1 - q^2|^2}.
	\]
\end{theorem}

The next result proves that our metric $G$ and Arcozzi-Sarfatti's metric, denoted here by $\widetilde{G}$, turn out to be exactly the same.

\begin{theorem}\label{thm:G=ArcozziSarfattiwidetildeG}
	The slice Riemannian metric $G$ of $\B$ satisfies
	\[
		\|\alpha\|^2_q = 
		G_q(\alpha,\alpha) =
			\frac{|\pi_I(\alpha)|^2}{(1 - |q|^2)^2}
				+ \frac{|\pi_I^\perp(\alpha)|^2}{|1 - q^2|^2},
	\]
	for every $q \in \B$ and $\alpha \in T_q\B$. In particular, $G = \widetilde{G}$, where the latter is the Riemannian metric considered in Theorem~\ref{thm:ArcozziSarfatti}.
\end{theorem}
\begin{proof}
	Once the identity in the statement has been proved, a polarization argument implies that $G = \widetilde{G}$. Hence, we proceed to establish the aforementioned identity.
	
	Let us fix $q \in \B$ and $\alpha \in T_q \B$. Note that each tangent space of $\B$ is canonically identified with $\HH$, and so we consider $\alpha \in \HH$. Let us assume that $q \in \C_I$, where $I \in \Sbb$, and let us choose $J \in \Sbb$ which anti-commutes with $I$. Thus, we can consider the decomposition
	\[
		\alpha = \alpha_1 + \alpha_2 J,
	\]
	where $\alpha_1, \alpha_2 \in \C_I$. Hence, $|\alpha_1|^2 = |\pi_I(\alpha)|^2$ and we also have
	\[
		|\alpha_2|^2 = -\pi_I^\perp(\alpha)^2, \quad
		|\alpha_2|^2 = |\pi_I^\perp(\alpha)|^2,
	\]
	where the first identity follows from the discussion at the end of the proof of Theorem~\ref{thm:SliceRiemannian}, and the second is a consequence of the first. Hence, by using these identities in the last formula for $G_q$ in Theorem~\ref{thm:SliceRiemannian} we obtain
	\begin{align*}
		\|\alpha\|^2_q =
		G_q(\alpha,\alpha) 
			&= \frac{|\alpha_1|^2 + |\alpha_2|^2}{(1 - |q|^2)^2}
				-4 \frac{|\im(q)|^2 |\alpha_2|^2}%
					{|1 - q^2|^2 (1 - |q|^2)^2} \\
			&= \frac{|\pi_I(\alpha)|^2 + |\pi_I^\perp(\alpha)|^2}%
							{(1 - |q|^2)^2}
				-4 \frac{|\im(q)|^2 |\pi_I^\perp(\alpha)|^2}%
					{|1 - q^2|^2 (1 - |q|^2)^2}  \\
			&= \frac{|\pi_I(\alpha)|^2}{(1 - |q|^2)^2}
				+ 	\bigg
					(\frac{|1 - q^2|^2 - 4 |\im(q)|^2}{(1 - |q|^2)^2}
					\bigg)
					\frac{|\pi_I^\perp(\alpha)|^2}{|1 - q^2|^2}.
	\end{align*}
	Hence, it is enough to prove that the quantity in parentheses is identically $1$ for every $q \in \B$. This is a simple computation that we show for the sake of completeness. Recall that $q \in \C_I$, so we write $q = x + y I$ with $x, y \in \R$. Then, we have
	\begin{align*}
		|1 - q^2|^2 - 4 |\im(q)|^2 
			&= |1 - (x^2 - y^2 + 2xyI)|^2 - 4y^2 \\
			&= (1 - x^2 + y^2)^2 + 4x^2y^2 - 4y^2 \\
			&= 1 + x^4 + y^4 - 2x^2 + 2y^2 - 2x^2y^2 
					+ 4x^2y^2 - 4y^2 \\
			&= 1 + x^4 + y^4 - 2x^2 - 2y^2 + 2x^2y^2 \\
			&= (1 - x^2 - y^2)^2 \\
			&= (1 - |q|^2)^2,
	\end{align*}
	which thus complete the proof.
\end{proof}

\begin{remark}\label{rmk:G=ArcozziSarfattiwidetildeG}
	A number of properties for the slice Riemannian metric $G$ can be derived from Theorem~\ref{thm:G=ArcozziSarfattiwidetildeG} thanks to the developments carried out in \cite{ArcozziSarfatti}. For the slice Riemannian metric $G = \widetilde{G}$, the latter computes the isometry group, the Lipschitz functions, geodesics and induced volume form on the boundary of $\B$, as well as their relationship with regular functions.
	
	On the other hand, it is important to emphasize the difference in techniques used to obtain $G = \widetilde{G}$ in \cite{ArcozziSarfatti} as opposed to our approach. The construction of this metric in \cite{ArcozziSarfatti} follows a function theoretic approach that mirrors the pseudo-hyperbolic distance in the complex unit disk as expressed in terms of a reproducing kernel. Our approach mirrors the complex case as well, which was described in section~\ref{sec:HermitianOnD}, but we do so from the viewpoint of transformations as geometric objects. Even though the quaternionic case of $\B$ does not entitle a group from regular M\"obius transformations, we have been able to extract enough group-like properties to construct the metric $G$ by looking for an invariant geometric tensor. One advantage of our method is the simplicity, in terms of computations needed, to obtain the slice Riemannian metric.
	
	At any rate, Theorem~\ref{thm:G=ArcozziSarfattiwidetildeG} ensures that both the function space theory and geometric/group-like approaches to construct a Riemannian metric lead to the same solution: the slice Riemannian metric $G$.
\end{remark}

We now prove that the slice Riemannian metric of $\B$ can be determined by its value at a single slice. This will make use of the non-regular M\"obius transformation $C_u$ defined by $C_u(q) = u^{-1} q u$, where $u \in \Spe(1)$.

\begin{corollary}[Riemannian representation formula]
	\label{cor:RiemRepFormula}
	The slice Riemannian metric is $C_u$-invariant for every $u \in \Spe(1)$. In particular, for any fixed $I_0 \in \Sbb$, the restriction of $G$ to the points of the slice $\D_{I_0}$ of $\B$ determines $G$ everywhere in the following sense. For every $I \in \Sbb$ and $u \in \Spe(1)$ such that $I = u^{-1} I_0 u$, we have
	\[
		G_q(\alpha, \beta) 
			= G_{uqu^{-1}} (u\alpha u^{-1},u\beta u^{-1}),
	\]
	for every $q \in \D_I$ and $\alpha, \beta \in \HH$, where we note that $u q u^{-1} \in \D_{I_0}$.
\end{corollary}
\begin{proof}
	Since $C_u$ is linear for every $u \in \Spe(1)$, it coincides with its differential. We also note that the following identities hold
	\begin{align*}
		\re(u \alpha u^{-1} \overline{u \beta u^{-1}})
			&= \re(\alpha \overline{\beta}) \\
		|1 - (uqu^{-1})^2| &= |1 - q^2| \\
		|uqu^{-1}| &= |q|,
	\end{align*}
	for every $q \in \B$, $\alpha, \beta \in \HH$ and $u \in \Spe(1)$. These remarks and the first formula for $G$ in Theorem~\ref{thm:SliceRiemannian} imply the claimed invariance. From this, the representation formula is just a restatement of the $C_{u^{-1}}$-invariance.
\end{proof}

The slice 2-form $\Omega$ has properties similar to those of $G$ that relates it to the classical complex case. As with the Riemannian case, we will denote by $\omega_I$ the K\"ahler form on $\D_I$ induced by the canonical identification $\D_I \simeq \D$.

\begin{theorem}\label{thm:SliceKahler}
	The slice 2-form $\Omega$ of $\B$ has the expression
	\[
		\Omega_q(\alpha, \beta) = 
		\frac{\im\big((1-q^2)^{-1}(\alpha - q\alpha q)%
			(\overline{\beta - q\beta q})%
				(1-\overline{q}^2)^{-1}\big)}{(1 - |q|^2)^2},
	\]
	for every $q \in \B$ and $\alpha, \beta \in \HH$. In particular, $\Omega$ satisfies the following properties.
	\begin{enumerate}
		\item (Anti-symmetry) $\Omega$ is an $\im(\HH)$-valued $2$-form of $\B$. More precisely, for every $q \in \B$, we have $\Omega_q(\alpha, \beta) = - \Omega_q(\beta, \alpha)$, for all $\alpha, \beta \in \HH$.
		\item (Non-degeneracy) $\Omega$ is non-degenerate. More precisely, for every $q \in \B$, if some $\alpha \in \HH$ satisfies $\Omega_q(\alpha, \beta) = 0$ for all $\beta \in \HH$, then $\alpha = 0$.
		\item (Slice closedness) For every $I \in \Sbb$, let us denote by $\Omega_I$ the restriction of $\Omega$ to the tangent bundle $T \D_I$ of $\D_I$. Then, $\Omega_I$ is an $I \R$-valued closed $2$-form. Furthermore, we have $\Omega_I = I \omega_I$.
	\end{enumerate}
\end{theorem}
\begin{proof}
	The expression for $\Omega_q$ in the statement is a direct consequence of the definition of $\Omega$ and Theorem~\ref{thm:SGeomStruct}.
	
	The anti-symmetry of $\Omega$ follows immediately from the property $\im(\alpha \overline{\beta}) = - \im(\beta \overline{\alpha})$ for $\alpha, \beta \in \HH$. On the other hand, the non-degeneracy of $\Omega$ follows from Definition~\ref{def:HermRiemSympRegularOnB}, the fact that $\dif\:(\cF_q)_q$ is an isomorphism (see Proposition~\ref{prop:FaRegularMobius}) and the non-degeneracy of the bilinear form $(\alpha, \beta) \mapsto \im(\alpha \overline{\beta})$.
	
	Finally, for a given $I \in \Sbb$, the restriction of $\Omega$ to the tangent bundle $T\D_I$ is obtained by considering the values 
	\[
		(\Omega_I)_q(\alpha, \beta) = \Omega_q(\alpha, \beta),
	\]
	only for $q \in \D_I$ and $\alpha, \beta \in \C_I$. In this case, we obtain from the above expression that
	\[
		(\Omega_I)_q(\alpha, \beta) 
			= \frac{\im(\alpha \overline{\beta})}{(1 - |q|^2)^2}.
	\]
	The latter expression is precisely $I$-times the K\"ahler form $\omega_I$ of the unit disk $\D_I$. In symbols we have $\Omega_I = I \omega_I$, and so	the slice closedness of $\Omega$, as stated in (3), follows from the closedness of $\omega_I$. 
\end{proof}

The previous result provides the basis for the next definition.

\begin{definition}\label{def:SliceKahler}
	The $\im(\HH)$-valued $2$-form $\Omega$ from Definition~\ref{def:HermRiemSympRegularOnB} will be called the slice K\"ahler form of $\B$.
\end{definition}

\begin{remark}\label{rmk:SliceKahler}
	To be absolutely clear about our notation above, recall that $\im$ stands for two different functions depending on whether it is applied on $\HH$ or $\C$. For the former, $\im$ takes values in the space generated by $\Sbb$, while for the latter $\im$ is $\R$-valued. Hence, $\Omega$ is vector-valued, but $\omega$ and $\omega_I$ are real-valued. This explains the need for the factor $I$ in the formula $\Omega_I = I \omega_I$ in Theorem~\ref{thm:SliceKahler}.
\end{remark}

As in the Riemannian case given by Corollary~\ref{cor:RiemRepFormula}, the value of the slice Hermitian metric and the slice K\"ahler form of $\B$ are completely determined by their values at a single slice. Note that the vector-valued nature of these tensors introduce an additional conjugation. In other words, we obtain equivariant geometric structures instead of invariant ones for the maps $C_u$, where $u \in \Spe(1)$.

\begin{corollary}[Slice Hermitian and slice K\"ahler representation formulas]
\label{cor:HermKahlerRepFormulas}	
	The slice Hermitian metric $H$ and the slice K\"ahler form $\Omega$ of $\B$ are $C_u$-equivariant for every $u \in \Spe(1)$. More precisely, we have 
	\begin{align*}
		H_q(\alpha, \beta) &= 
			u^{-1}\big( 
				H_{uqu^{-1}}(u\alpha u^{-1}, u\beta u^{-1}) 
			\big)u, \\
		\Omega_q(\alpha, \beta) &= 
			u^{-1}\big( 
				\Omega_{uqu^{-1}}(u\alpha u^{-1}, u\beta u^{-1}) 
			\big)u,
	\end{align*}
	for every $u \in \Spe(1)$, $q \in \B$ and $\alpha, \beta \in \HH$. In particular, the values of $H$ and $\Omega$ at the points of any slice can be recovered from the values at the points of a given single slice using the formulas above.
\end{corollary}
\begin{proof}
	The equivariance formulas stated are easily obtained by substitution on the expressions for $H$ and $\Omega$ given in Theorems~\ref{thm:SGeomStruct} and \ref{thm:SliceKahler}, respectively.
\end{proof}

\begin{remark}[A comparison with other quaternionic-like structures]
\label{rmk:ComparisonHyperKahler}
	With the previous results at hand, we compare the $\Spe(1,1)$-invariant metric $\widehat{G}$ and a couple of well known geometric quaternionic structures with our slice geometric ones. For simplicity, we will discuss the $4$-dimensional case which is clearly enough for our purposes.
	
	In the first place, we have the notion of a hyper-K\"ahler structure. On a $4$-dimensional manifold, such structure is given by a Riemannian metric and three integrable complex structures orthogonal for the metric and satisfying the usual quaternionic commutation relations. Further details on their definition and properties can be found in \cite{BesseEinstein}. As stated in Theorem~14.13 of the latter reference, Berger proved that every hyper-K\"ahkler manifold is Ricci flat. However, the metric $\widehat{G}$ is known to have strictly negative sectional curvature (see \cite{Helgason}) and so strictly negative Ricci curvature. Hence, from Theorem~\ref{thm:GeometryMobiusNonRegular} it follows that there is no hyper-K\"ahler structure on $\B$ that is invariant under the $\Spe(1,1)$-action by non-regular M\"obius transformations. 
	
	In second place, there is the notion of quaternion-K\"ahler manifold. We refer again to \cite{BesseEinstein} for further details. One possible definition, in the $4$-dimensional case, is obtained by requiring the given Riemannian manifold to have holonomy $\Spe(1) \times \Spe(1)$. This is in fact the case for the metric $\widehat{G}$, since the holonomy and the isotropy coincide for Riemannian symmetric spaces (see \cite{Helgason}). Hence, $\B$ with the metric $\widehat{G}$ is quaternion-K\"ahler. It is also known, as discussed in section~D from Chapter~14 of \cite{BesseEinstein} and in page~92 of \cite{HarveySpinors}, that every quaternion-K\"ahler manifold carries three locally defined almost-complex structures and three locally defined $2$-forms. The latter allow to construct a globally defined $4$-form. However, in the case of $\B$ with the metric $\widehat{G}$ the locally defined objects cannot be extended globally. In particular, the quaternion-K\"ahler structure obtained from the $\Spe(1,1)$-action does not yield associated globally defined $2$-forms.
	
	In conclusion, the $\Spe(1,1)$-action by non-regular M\"obius transformations fails to provide a notion of K\"ahler structure on $\B$ as some sort of $2$-form involving $\HH$ that could also be nicely related to a corresponding Riemannian metric.
	
	However, the slice regular K\"ahler structure $\Omega$ does provide a natural $2$-form with the K\"ahler-like properties obtained in Theorem~\ref{thm:SliceKahler}. Furthermore, by construction, $\Omega$ exhibits as much invariance under $\cM(\B)$ as the latter allows considering that it is not a group. Since, $\Omega$ is $\im(\HH)$-valued we can always consider the components with respect to an orthonormal basis of $\im(\HH)$ to obtain three real-valued $2$-forms. The slice Hermitian metric $H$ and the slice Riemannian metric $G$ have similar properties as stated in Theorems~\ref{thm:SGeomStruct} and \ref{thm:SliceRiemannian}, respectively, and their corollaries. More precisely, they both exhibit as much invariance as $\cM(\B)$ allows.
	
	Also, a certain degree of invariance for all three slice geometric structures has been obtained in Corollaries~\ref{cor:RiemRepFormula} and \ref{cor:HermKahlerRepFormulas}. The latter have provided the sort of Representation Formulas found in the theory of slice regular functions, but now for geometric structures on $\B$ built from regular M\"obius transformations.
	
	Finally, the slice Riemannian metric $G$ and the slice K\"ahler form $\Omega$ can be considered nicely related in a similar way as it occurs in the classical complex case. The reason is that $G$ and $\Omega$ are the (quaternionic) real and imaginary parts, respectively, of the slice Hermitian metric $H$. With all three of them built out of the same principle: regular M\"obius transformations.
\end{remark}

\subsection*{Acknowledgment}
This research was partially supported by SNI-Conahcyt and by Conahcyt Grants 280732 and 61517.

\subsection*{Data Availability} 
Data sharing not applicable to this article as no datasets were generated or analyzed during the current study.

\subsection*{Competing interests} The author has no competing interests to declare that are relevant to the content of this article.


\begin{thebibliography}{XX}

\bibitem{ArcozziSarfatti} Arcozzi, Nicola and Sarfatti, Giulia:
\emph{Invariant metrics for the quaternionic Hardy space}. J. Geom. Anal. 25 (2015), no.3, 2028--2059.

\bibitem{BesseEinstein} Besse, Arthur L.: Einstein manifolds.
Reprint of the 1987 edition Classics Math. Springer-Verlag, Berlin, 2008.

\bibitem{BisiGentiliMobius} Bisi, Cinzia and Gentili, Graziano: \emph{M\"obius transformations and the Poincar\'e distance in the quaternionic setting}. Indiana Univ. Math. J. 58 (2009), no.6, 2729--2764.

\bibitem{BisiGentiliGeometryHUnitDisc} Bisi, Cinzia and Gentili, Graziano: \emph{On the geometry of the quaternionic unit disc}. Hypercomplex analysis and applications, 1–11. Trends Math. Birkh\"auser/Springer Basel AG, Basel, 2011.

\bibitem{BisiStoppatoSchwarz} Bisi, Cinzia and Stoppato, Caterina:
\emph{The Schwarz-Pick lemma for slice regular functions}. Indiana Univ. Math. J. 61 (2012), no.1, 297–-317.

\bibitem{BisiStoppatoMobius} Bisi, Cinzia and Stoppato, Caterina: \emph{Regular vs.~classical M\"obius transformations of the quaternionic unit ball}. Advances in hypercomplex analysis, 1–13.
Springer INdAM Ser., 1 Springer, Milan, 2013.

\bibitem{CaoParkerWang} Cao, Wensheng; Parker, John R. and Wang, Xiantao: \emph{On the classification of quaternionic M\"obius transformations}. Math. Proc. Cambridge Philos. Soc. 137 (2004), no.2, 349--361.

\bibitem{ChenGreenberg} Chen, S. S. and Greenberg, L.: \emph{Hyperbolic spaces}. Contributions to analysis (a collection of papers dedicated to Lipman Bers), pp. 49--87 Academic Press [Harcourt Brace Jovanovich, Publishers], New York-London, 1974

\bibitem{ColomboSabadiniEtAlSliceMonogenic} Colombo, Fabrizio; Kimsey, David P.; Pinton, Stefano and  Sabadini, Irene: \emph{Slice monogenic functions of a Clifford variable via the S-functional calculus}.
Proc. Amer. Math. Soc. Ser. B 8 (2021), 281--296.

\bibitem{ColomboKraussharSabadiniSymmetries} Colombo, Fabrizio; Krau{\ss}har, Rolf S\"oren and Sabadini, Irene: \emph{Symmetries of slice monogenic functions}. J. Noncommut. Geom. 14 (2020), no.3, 1075--1106.

\bibitem{ColomboSabadiniStruppaFunctionalBook} Colombo, Fabrizio; Sabadini, Irene and Struppa, Daniele C.: Noncommutative functional calculus.
Theory and applications of slice hyperholomorphic functions. Progr. Math., 289
Birkh\"auser/Springer Basel AG, Basel, 2011.

\bibitem{GentiliStoppatoStruppa2ndEd} Gentili, Graziano; Stoppato, Caterina and Struppa, Daniele C.:Regular functions of a quaternionic variable. Second edition. Springer Monogr. Math. Springer, Cham, 2022.

\bibitem{GentiliStruppa2007} Gentili, Graziano and Struppa, Daniele C.: \emph{A new theory of regular functions of a quaternionic variable.} Adv. Math. 216 (2007), no.1, 279--301.

\bibitem{HarveySpinors} Harvey, F. Reese: Spinors and calibrations. Perspect. Math., 9 Academic Press, Inc., Boston, MA, 1990.

\bibitem{Helgason} Helgason, Sigurdur: Differential geometry, Lie groups, and symmetric spaces. Corrected reprint of the 1978 original. Grad. Stud. Math., 34 American Mathematical Society, Providence, RI, 2001.

\bibitem{StoppatoMobius} Stoppato, Caterina: \emph{Regular Moebius transformations of the space of quaternions}. Ann. Global Anal. Geom. 39 (2011), no.4, 387--401.

\end{thebibliography}
\end{document}